\documentclass[11pt,a4paper]{amsart}
\usepackage[T1]{fontenc}
\usepackage[utf8]{inputenc}
\usepackage[english]{babel}
\usepackage{amsmath,amssymb,amsthm,mathtools}
\usepackage[margin=2.6cm]{geometry}
\usepackage[expansion=false]{microtype}
\usepackage{needspace}
\usepackage[hidelinks]{hyperref}
\hypersetup{pdftitle={Every introenumerable set contains a uniformly introreducible subset},pdfauthor={Patrizio Cintioli}}
\newtheorem{theorem}{Theorem}[section]
\newtheorem{lemma}[theorem]{Lemma}
\newtheorem{proposition}[theorem]{Proposition}
\newtheorem{corollary}[theorem]{Corollary}
\theoremstyle{definition}
\newtheorem{definition}[theorem]{Definition}
\newtheorem*{questionGHPT}{Question 1.7}
\theoremstyle{remark}
\newtheorem{remark}[theorem]{Remark}
\newcommand{\N}{\omega}
\newcommand{\ui}{\leq_T^{\mathrm{ui}}}
\newcommand{\intr}{\leq_T^{\mathrm i}}
\newcommand{\uie}{\leq_e^{\mathrm{ui}}}
\newcommand{\ice}{\leq_{\mathrm{c.e.}}^{\mathrm{i}}}
\newcommand{\uice}{\leq_{\mathrm{c.e.}}^{\mathrm{ui}}}
\newcommand{\Graph}{\operatorname{Graph}}
\newcommand{\rk}{\operatorname{rk}}
\newcommand{\ot}{\operatorname{ot}}
\newcommand{\rng}{\operatorname{rng}}
\newcommand{\HYP}{\operatorname{HYP}}
\newcommand{\FF}{\mathcal F}
\newcommand{\GG}{\mathcal G}
\newcommand{\HH}{\mathcal H}
\newcommand{\KK}{\mathcal K}
\newcommand{\TT}{\mathcal T}
\newcommand{\OO}{\mathcal O}
\newcommand{\res}{\upharpoonright}

\title[Uniformly introreducible subsets of introenumerable sets]
{Every introenumerable set contains a uniformly introreducible subset}
\author{Patrizio Cintioli}

\address{Mathematics Division, School of Science and Technology, University of Camerino, Italy}
\email{patrizio.cintioli@unicam.it}


\begin{document}

\begin{abstract}
We prove that every infinite introenumerable set contains an
infinite uniformly introreducible subset, answering both parts
of Question 1.7 of Greenberg, Harrison-Trainor, Patey, and
Turetsky affirmatively.
Moreover, a single procedure computes the original set from
every infinite subset of the selected set.
The proof combines the uniformization and refinement results
of these authors with Ramsey-theoretic methods and a canonical
reconstruction argument.
We also obtain refinements controlling the ordinals and
hyperjumps of the selected subsets.
\end{abstract}

\maketitle

\section{Introduction}

An infinite set $A$ is introreducible if its characteristic function
is computable from every infinite subset of $A$. The information
carried by the set is
therefore preserved under arbitrary infinite thinning, although the
procedure that recovers it may depend on the subset. Uniform
introreducibility requires one procedure to work for all infinite
subsets. These two notions need not coincide: Lachlan's construction,
presented in \cite[Theorem 4.1]{Jockusch}, gives an introreducible set
that is not uniformly introreducible.

A complementary direction is the study of introimmunity: for a
specified reducibility, an introimmune set cannot be recovered from
any subset omitting infinitely many of its elements.
Recent existence results for introimmune sets with respect to several
reducibilities are obtained in \cite{CintioliIntroimmunity}.

A standard example of a uniformly introreducible set is given by
the family $F_A$ of all finite initial segments of the characteristic
function of a set $A$ of natural numbers. Here finite binary strings
are identified with natural numbers using a fixed computable
bijection. It is easy to see that $F_A$ is uniformly introreducible.
Moreover, $A$ and $F_A$ are Turing equivalent.
Thus every Turing degree contains a uniformly introreducible set;
see the discussion of Dekker sets in \cite[Section 1]{GHPT}.

This recoding, however, need not produce a subset of $A$. The problem
studied here is instead whether one can obtain uniform
introreducibility by selecting an infinite subset of the original set.

The corresponding enumeration notion provides a broader setting for
this problem. An infinite set is introenumerable if it is c.e. relative
to each of its infinite subsets, and uniformly introenumerable if one
relative c.e. operator works for all of them. Jockusch
\cite[p.~535]{Jockusch} asked whether every infinite uniformly
introenumerable set contains an infinite uniformly introreducible
subset. Greenberg, Harrison-Trainor, Patey, and Turetsky
\cite[Theorem 1.4]{GHPT} proved that it does. They also posed the
following two-part question.

\begin{questionGHPT}[\cite{GHPT}]
Does every infinite introreducible set have an infinite uniformly
introreducible subset? Does every infinite introenumerable set have
an infinite uniformly introenumerable subset?
\end{questionGHPT}

Subsequent work of Em gives a related sufficient condition.
In \cite[Proposition 4.20]{Em}, an introreducible set $A$ is shown
to have a uniformly introreducible subset when it is finitely
decomposable, that is, when
\[
  A=A_0\cup\cdots\cup A_n
\]
for some infinite sets $A_0,\ldots,A_n$ such that
$A\leq_T^{\mathrm{ui}} A_i$ for every $i\leq n$.
Our result does not establish this finite-decomposition condition for
arbitrary introreducible sets. Rather, the proof bypasses it and uses
a different strategy based on uniformization, Ramsey-theoretic
certification, and canonical reconstruction.

We prove that every infinite introenumerable set contains an infinite
uniformly introreducible subset. This gives affirmative answers to
both parts of \cite[Question 1.7]{GHPT}. In fact, the construction
provides a nonempty class of such subsets with a common reconstruction
procedure, and every infinite subset of any member of the class also
computes the original set $A$ by a second common procedure. The precise
statement includes a definability bound on this class, which permits
separate applications of basis theorems to control the ordinal and
hyperjump of a selected member.

\medskip
\paragraph{\textit{Statement and conventions}}

We write $\N=\{0,1,2,\ldots\}$.
For a set $B\subseteq\N$, let
\[
  [B]^\omega=\{D\subseteq B:D\text{ is infinite}\}.
\]
We also write $[B]^{<\omega}$ for the family of all finite
subsets of $B$, including the empty set, and
$[B]^k=\{F\subseteq B:|F|=k\}$ for $k\in\N$.

For $A\subseteq\N$, we identify $A$ with its characteristic
function and write
\[
  \Graph(A)=\{(n,A(n)):n\in\N\},
\]
using a computable coding of pairs.
We write $A\leq_T B$ if the characteristic function of $A$
is computable with oracle $B$.
Fix an effective enumeration $(\Phi_e)_{e\in\N}$ of the
ordinary Turing functionals.
The notation $\Phi_e^D=A$ means that $\Phi_e^D$ is total
and agrees with the characteristic function of $A$.

For $A\subseteq\N$ and an infinite set $B\subseteq\N$, define
\[
\begin{aligned}
  A\intr B
  &\iff
  (\forall D\in[B]^\omega)\; A\leq_T D\\
  &\iff
  (\forall D\in[B]^\omega)(\exists e\in\N)\;
  \Phi_e^D=A,
\end{aligned}
\]
and
\[
  A\ui B
  \iff
  (\exists e\in\N)(\forall D\in[B]^\omega)\;
  \Phi_e^D=A.
\]
Thus $A\intr B$ means that every infinite subset of $B$
computes $A$, with the reduction index allowed to depend
on that subset.
In contrast, $A\ui B$ requires a single index that works
for every infinite subset of $B$.

An infinite set $A$ is called \emph{introreducible} if
$A\intr A$, and \emph{uniformly introreducible} (UI) if
$A\ui A$.
Explicitly,
\[
  A\ui A
  \iff
  (\exists e\in\N)(\forall D\in[A]^\omega)
  (\forall n\in\N)\;
  \Phi_e^D(n)\downarrow=A(n).
\]
The witnessing index may depend on $A$, but it must be
independent of the infinite subset $D$.

For an infinite set $B$, write $X\ice B$ when $X$ is c.e. relative
to every member of $[B]^\omega$, and $X\uice B$ when one relative
c.e. operator works for all these oracles. Write $X\uie B$ when
one ordinary positive enumeration operator enumerates $X$ from
every member of $[B]^\omega$.
An infinite set $A$ is introenumerable if $A\ice A$, and uniformly
introenumerable if $A\uice A$.
For infinite sets, $X\uice B$ and $X\uie B$ are equivalent by
\cite[Proposition 2.1]{GHPT}.
Every introreducible set is introenumerable, and every uniformly
introreducible set is uniformly introenumerable.

Every functional described as ordinary receives only the oracles
explicitly indicated. All quantifiers over infinite subsets are
interpreted in the ambient universe.

We use the standard lightface notation $\Sigma^1_1(A)$ and
$\Delta^1_1(A)$, relativized to the oracle $A$; see either
\cite[Chapters~4--5]{Montalban} or \cite[Section~IV.2]{Odifreddi1989}
for background.

The following theorem gives the common refinement statement discussed
above. The two functionals have distinct targets: one recovers the
original set $A$, and the other recovers the selected subset $R$.

\begin{theorem}\label{thm:main}
Every infinite introenumerable set $A$ contains an infinite uniformly
introreducible subset.
More precisely, there exist a nonempty $\Sigma^1_1(A)$ class
$\KK\subseteq[A]^\omega$ and two ordinary functionals
$\Psi,\Theta$ such that
\begin{equation}\label{eq:main}
  \forall R\in\KK\ \forall D\in[R]^\omega\qquad
  \Psi^D=A\quad\text{and}\quad\Theta^D=R.
\end{equation}
\end{theorem}

The affirmative answers to the two parts of Question 1.7 are stated
separately in Corollaries~\ref{cor:first} and~\ref{cor:second} below.
We do not assume that the final subset is hyperarithmetic in $A$
or has the same Turing degree as $A$.
A single decoder index works for all infinite subsets of a given $R$,
and indeed for every $R$ in the fixed class. We do not claim that the
indices can be chosen computably from an arbitrary introenumerable
set $A$.

For an infinite set $D$, let $p_D$ be its increasing enumeration,
indexed starting at zero. We identify finite sets with their
increasing enumerations and write $E\preceq D$ for the initial-segment
relation. We set $\max\varnothing=-1$ and
$D_{>x}=D\cap(x,\infty)$.
We write $\HYP(Z)=\Delta^1_1(Z)$ for the sets hyperarithmetic in $Z$,
$\omega_1^Z$ for the least ordinal without a $Z$-computable
presentation, and $\OO^Z$ for the hyperjump of $Z$.

For a nonempty well-founded tree $T\subseteq\N^{<\omega}$,
we define the rank of each node $\sigma\in T$ by
\[
  \rk_T(\sigma)=
  \sup\{\rk_T(\sigma{}^\frown\langle x\rangle)+1:
  x\in\N,\ \sigma{}^\frown\langle x\rangle\in T\},
  \qquad
  \rk(T)=\rk_T(\varnothing).
\]
Here $\sup\varnothing=0$, so a leaf has rank zero.
The rank of a rooted subtree is at most that of the containing tree.

\medskip
\paragraph{\textit{Proof strategy and organization.}}
The distinction between a fixed target and a selected subset is
central to the argument. The uniformization theorem of GHPT
\cite[Theorem 1.2]{GHPT} turns $A\intr B$ into $A\ui C$ on some
infinite $C\subseteq B$. When $B=A$, this still computes $A$, not
$C$. If in addition $C\in\Delta^1_1(A)$, a further refinement gives
a uniformly introreducible subset \cite[Lemma 2.2]{GHPT}.
Our local criterion, Theorem~\ref{thm:local}, requires only
$\omega_1^C\leq\omega_1^A$ in addition to $A\ui C$ and
$C\subseteq A$; its output is a subset of $A$, not necessarily of $C$.

Section~\ref{sec:uniformization} combines the uniformization and
refinement results of GHPT
\cite[Proposition 2.1, Corollary 3.13, and Proposition 5.3]{GHPT}
with preservation of computable ordinals under hyperarithmetic
reducibility, by Spector's theorem \cite{Spector};
see also \cite[Chapter~5]{Montalban}.

For an infinite introenumerable $A$, it gives such a set $C$ and
an ordinary positive operator that enumerates $\Graph(A)$ from
every infinite subset of $C$. Failures to enumerate an answer at
a given input are represented by omission trees. Combining these
in a single well-founded tree gives one rank bound below
$\omega_1^A$ for all inputs simultaneously.

Sections~\ref{sec:fronts} and~\ref{sec:nucleus} convert this bound
into finite certificates. We construct an $A$-decidable front whose
rank exceeds the omission bound on every infinite restriction.
A binary labeling tests whether a finite approximation is correct
and supplies answers on a specified finite interval.
The clopen consequence of the Nash--Williams theorem
\cite[Theorem 1.2]{MV}, together with the rank comparison, forces a
nonempty homogeneous side on which these certificates are available.
A decoder reads finite approximations until an answer appears;
correctness is certified without requiring the decoder to recognize
the front or consult $A$.

Section~\ref{sec:profiles} addresses reconstruction of the selected
subset. For a nonempty homogeneous side of a clopen coloring,
relative to an oracle $Z$, we construct coherent color profiles
on finite sets. A profile determines a canonical homogeneous set
$R$, and $Z$ together with any infinite $D\subseteq R$ recovers
the whole profile. Repeating the canonical construction then
recovers $R$. Arithmetical conditions on profiles and their witnesses
yield a nonempty $\Sigma^1_1(Z)$ class of outputs with one common
decoder.

In Section~\ref{sec:conclusion}, we apply this relative construction
with $Z=A$. The decoder for $A$ obtained from finite certificates
supplies the oracle needed to reconstruct $R$, giving an ordinary
decoder with oracle $D$ alone. This proves the local criterion and
Theorem~\ref{thm:main}. The two parts of Question 1.7 follow as
Corollaries~\ref{cor:first} and~\ref{cor:second}.
Separate applications of the Gandy basis theorem and the relativized
hyperjump basis theorem \cite[Corollary 2.7]{JS} give
Corollaries~\ref{cor:ordinal} and~\ref{cor:hyperjump}.
Section~\ref{sec:conclusions} concludes with open questions about
the complexity of the selected subset and finite decomposability.

\section{Uniformization for introenumerable sets and a single omission tree}
\label{sec:uniformization}

This section supplies the initial uniform computation of $A$ and
an ordinal bound on its possible failures at finite stages.
The intermediate set $C$ is not yet required to compute itself
uniformly from its infinite subsets. What matters here is that
$A\ui C$ and that the associated omission rank is below
$\omega_1^A$.
We first separate the source results for relative enumeration,
positive enumeration, and computation of a characteristic function.

\medskip
\paragraph{\textit{Uniform enumeration without ordinal increase.}}
By \cite[Corollary 3.13]{GHPT}, for infinite sets $X,B$,
\begin{equation}\label{eq:ghpt-uniformization}
  X\ice B\ \Longrightarrow\
  \exists B_0\in[B]^\omega\,
  \bigl(X\uice B_0\ \land\ \omega_1^{B_0}\leq\omega_1^B\bigr).
\end{equation}

\medskip
\paragraph{\textit{Positive compilation.}}
For infinite sets $X,B$, \cite[Proposition 2.1]{GHPT} gives
\begin{equation}\label{eq:positive}
  X\uice B\quad\Longleftrightarrow\quad X\uie B.
\end{equation}
The conversion is an effective transformation of ordinary operator
indices. It neither adds an oracle nor changes the underlying set.

\medskip
\paragraph{\textit{Improvement for an introenumerable target.}}
By \cite[Proposition 5.3]{GHPT}, if $B$ is infinite, then
\begin{equation}\label{eq:ghpt-improvement}
  \left.
  \begin{array}{l}
    X\text{ is infinite and introenumerable},\\
    X\uie B
  \end{array}
  \right\}
  \Longrightarrow
  \exists C\in[B]^\omega\cap\Delta^1_1(B)\ (X\ui C).
\end{equation}
The introenumerability hypothesis concerns the target $X$, not the
underlying set $B$. The conclusion is an ordinary uniform Turing
reduction. The complexity bound is relative to $B$, not to $X$.

\begin{lemma}[Hyperarithmetic reducibility and computable ordinals]
\label{lem:ordinal}
If $X\in\Delta^1_1(Y)$, then
\[
 \omega_1^X\leq\omega_1^Y,
 \qquad
 \omega_1^{Y\oplus X}=\omega_1^Y.
\]
\end{lemma}

\begin{proof}
Put $Z=Y\oplus X$. Since $X\in\Delta^1_1(Y)$, closure
under finite joins gives $Z\in\Delta^1_1(Y)$.

Let $\alpha<\omega_1^Z$, and choose a $Z$-computable
presentation of a well-order of order type $\alpha$.
Its code is hyperarithmetic in $Y$, by downward closure
of $\Delta^1_1(Y)$ under Turing reducibility.
By Spector's theorem, relativized to $Y$ \cite{Spector}
(see also \cite[Chapter~5]{Montalban}), $\alpha$ has a
$Y$-computable presentation. Hence $\alpha<\omega_1^Y$.
It follows that $\omega_1^Z\leq\omega_1^Y$.

The reverse inequality follows from $Y\leq_T Z$,
so $\omega_1^{Y\oplus X}=\omega_1^Y$.
Finally, $X\leq_T Z$ gives
$\omega_1^X\leq\omega_1^Z=\omega_1^Y$.
\end{proof}

\medskip
\paragraph{\textit{Why the improvement supplies an ordinary decoder.}}
We recall the mechanism of \cite[Proposition 5.3]{GHPT}; its
substantive input, Proposition 5.2, remains an external result.
Suppose $A$ is introenumerable and $A\uie B$.
By \cite[Proposition 5.2]{GHPT}, choose infinite $X\subseteq A$ and
$Y\subseteq B$ such that
\[
 X,Y\in\Delta^1_1(B),\qquad X\ui Y.
\]
Fix an ordinary functional $\Phi$ witnessing the last reduction.
Since $A$ is introenumerable and $X\subseteq A$ is infinite, choose an
ordinary index $e$ with $W_e^X=A$. This is a choice of one natural-number
index for the fixed set $X$; it does not assert uniformity for arbitrary
subsets of $A$.

Let $(W_{e,s}^X)$ be a finite increasing $X$-computable approximation and put
\[
 g(n)=\min\{s:W_{e,s}^X\cap[0,n]=A\cap[0,n]\}.
\]
This is a settling-time function and $g\leq_T X'$. Indeed, $X'$ decides
membership in $A=W_e^X$, and then finds stages at which the finitely many
positive entries up to $n$ have appeared. Hence $g\in\Delta^1_1(B)$.
Choose a strictly increasing sequence $(c_n)$ in $Y$ with $c_n\geq g(n)$,
computably in $Y\oplus g$, and put $C=\{c_n:n\in\omega\}$.
The increasing enumeration also decides $C$, so $C\in\Delta^1_1(B)$.

For every $D\in[C]^\omega$, the ordinary procedure $\Phi^D$ computes $X$,
and
\[
 p_D(n)\geq p_C(n)=c_n\geq g(n).
\]
An ordinary decoder for $A(n)$ simulates the finite approximation
$W_{e,p_D(n)}^X$ using $\Phi^D$ to answer its oracle queries, and returns
$1$ exactly when $n$ belongs to that approximation. This terminates and
returns $A(n)$. Neither $g$, $X'$, $A$, nor $B$ is supplied as an additional
oracle to the decoder. The sparse choice of $C$ makes $p_D$ supply a
sufficient numerical time bound.

\begin{lemma}[An ordinal-preserving consequence of GHPT]\label{lem:ghpt}
If $A$ is infinite and introenumerable, there exist $C\in[A]^\omega$
and an ordinary positive enumeration operator $\Gamma$ such that
$A\ui C$ and
\begin{equation}\label{eq:uniformizer}
  \omega_1^C=\omega_1^A,
  \qquad
  \Gamma(D)=\Graph(A)\quad(D\in[C]^\omega).
\end{equation}
\end{lemma}
\begin{proof}
Since $A\ice A$, \eqref{eq:ghpt-uniformization} gives an infinite
$B\subseteq A$ such that
\[
  A\uice B,\qquad \omega_1^B\leq\omega_1^A.
\]
By \eqref{eq:positive}, $A\uie B$.
Since the target $A$ is introenumerable,
\eqref{eq:ghpt-improvement} now gives
\[
  C\in[B]^\omega\cap\Delta^1_1(B),\qquad A\ui C.
\]
Lemma~\ref{lem:ordinal} yields
\[
  \omega_1^C\leq\omega_1^B\leq\omega_1^A.
\]
Taking $D=C$ in $A\ui C$, we obtain $A\leq_T C$ and hence
$\omega_1^A\leq\omega_1^C$.
Thus $\omega_1^C=\omega_1^A$, and $C\subseteq B\subseteq A$.

Finally, fix an ordinary Turing functional $\Psi$ witnessing
$A\ui C$. The ordinary relative c.e. operator which, with
oracle $D$, enumerates $(n,b)$ whenever
$\Psi^D(n)\downarrow=b\in\{0,1\}$ witnesses
\[
  \Graph(A)\uice C.
\]
Since both $\Graph(A)$ and $C$ are infinite,
\cite[Proposition 2.1]{GHPT} yields an ordinary positive
enumeration operator $\Gamma$ such that
\[
  \Gamma(D)=\Graph(A)\qquad(D\in[C]^\omega).
\]
Together with $\omega_1^C=\omega_1^A$, this proves
\eqref{eq:uniformizer}.
\end{proof}

\begin{remark}\label{rem:bridge-scope}
The first use of positive compilation enumerates $A$, not
$\Graph(A)$.
The passage to the graph occurs only after
\cite[Proposition 5.3]{GHPT} has provided an ordinary total
Turing decoder for $A$.
Also, $C\in\Delta^1_1(B)$ does not imply $C\in\Delta^1_1(A)$:
the first uniformizing set $B$ was not asserted to be hyperarithmetic
in $A$. Only its ordinal was controlled.
The lemma alone does not show that $C$ is uniformly introreducible:
the target of $A\ui C$ is $A$, not $C$.
\end{remark}

\medskip
\paragraph{\textit{The introreducible special case.}}
When $A$ is introreducible, the uniformizing set can instead be
obtained directly from \cite[Corollary 3.13]{GHPT} with target
$\Graph(A)$ and underlying set $A$.
Indeed, every infinite subset of $A$ computes $A$ and hence
enumerates its graph. The corollary and positive compilation give
$C\in[A]^\omega$ and an ordinary $\Gamma$ with
$\omega_1^C\leq\omega_1^A$ and
$\Gamma(D)=\Graph(A)$ for every $D\in[C]^\omega$.
Enumerating $\Gamma(C)$ computes $A$, so the reverse ordinal
inequality follows.
This direct route does not use the introenumerable-target
improvement in \cite[Proposition 5.3]{GHPT}.

We recall explicitly why the positive compilation is safe.
If a computation on a finite binary string $\sigma$ outputs $(n,b)$,
use the finite set
$F_\sigma=\{u<|\sigma|:\sigma(u)=1\}$ as its support.
When $F_\sigma\subseteq C$, the infinite oracle
\[
F_\sigma\cup\bigl(C\setminus[0,|\sigma|)\bigr)
\]
is a subset of $C$ and extends $\sigma$. Correctness on all infinite
subsets forces $b=A(n)$. Thus discarding negative oracle answers
introduces no incorrect outputs on supports contained in $C$.
The computations on each infinite $D\subseteq C$ ensure that the
compiled operator has enough axioms to enumerate the entire graph
from that same $D$.

For the graph-enumeration operators used below, fix finite
approximations $\Gamma_s$ obtained from the first $s$
stages of the enumeration of axioms, with $\Gamma_0(F)=\varnothing$.
For finite $F$, the set $\Gamma_s(F)$ is finite and computable, and
\begin{equation}\label{eq:monotonicity}
s\leq t,\ F\subseteq E
\quad\Longrightarrow\quad
\Gamma_s(F)\subseteq\Gamma_t(E).
\end{equation}
We may restrict outputs to $\N\times\{0,1\}$.

For every set $D\subseteq\N$, these approximations satisfy
\[
  \Gamma(D)=
  \bigcup_{s\in\N}\;
  \bigcup_{F\in[D]^{<\omega}}\Gamma_s(F).
\]

For each input, an omission tree records finite subsets on which
neither possible answer has yet appeared. We place all these trees
under a common root so that their ranks have a single bound.
This bound, rather than the individual component ranks, will be
compared with the rank of a front in Section~\ref{sec:nucleus}.

\begin{lemma}[The aggregate tree]\label{lem:rank}
Suppose that $C$ is infinite and that $\Gamma$ is an ordinary
positive enumeration operator such that
$\Gamma(D)=\Graph(A)$ for every $D\in[C]^\omega$.
For $n\in\N$, let
\begin{equation}\label{eq:omission}
T_n(C)=\left\{\tau:\begin{array}{l}
\tau\text{ is strictly increasing},\ \rng(\tau)\subseteq C,\\
(n,0),(n,1)\notin\Gamma_{|\tau|}(\rng(\tau))
\end{array}\right\}.
\end{equation}
The tree
\begin{equation}\label{eq:aggregate}
\TT_\Gamma(C)=\{\varnothing\}\cup
\{\langle n\rangle{}^\frown\tau:n\in\N,\ \tau\in T_n(C)\}
\end{equation}
is $C$-computable and well-founded. Consequently,
\begin{equation}\label{eq:rank-bound}
\rho:=\rk(\TT_\Gamma(C))<\omega_1^C,
\qquad \rk(T_n(C))<\rho\quad(n\in\N).
\end{equation}
\end{lemma}

\begin{proof}
Since $\Gamma_0(\varnothing)=\varnothing$, the empty sequence
belongs to $T_n(C)$ for every $n\in\N$.
If $\sigma\preceq\tau\in T_n(C)$, then
\[
 \Gamma_{|\sigma|}(\rng(\sigma))
 \subseteq
 \Gamma_{|\tau|}(\rng(\tau))
\]
by \eqref{eq:monotonicity}.
Thus $\sigma\in T_n(C)$, so each $T_n(C)$ is a tree.
Its membership relation is decidable in $C$, uniformly in $n$:
all oracle queries concern the finitely many entries of the
given sequence, and the remaining test concerns a finite
computable approximation to $\Gamma$.

Fix $n$ and suppose, towards a contradiction, that
$f\in[T_n(C)]$.
Then $f$ is strictly increasing and
$D=\rng(f)$ is an infinite subset of $C$.
Since $\Gamma(D)=\Graph(A)$, there exist a finite set
$F\subseteq D$ and a stage $s$ such that
\[
 (n,A(n))\in\Gamma_s(F).
\]
Choose $m\geq s$ sufficiently large that
$F\subseteq\rng(f\res m)$.
By monotonicity,
\[
 (n,A(n))\in\Gamma_m(\rng(f\res m)),
\]
contrary to $f\res m\in T_n(C)$.
Hence every $T_n(C)$ is well-founded.

Every nonempty finite sequence has a unique representation
$\langle n\rangle{}^\frown\tau$.
The uniform decidability of the component trees therefore
shows that $\TT_\Gamma(C)$ is a $C$-computable tree.
An infinite path through it would have a first entry $n$
and a tail belonging to $[T_n(C)]$, which is impossible.
Thus $\TT_\Gamma(C)$ is well-founded.

The immediate successors of its root are exactly the nodes
$\langle n\rangle$, for $n\in\N$.
For each $n$, the map
$\tau\mapsto\langle n\rangle{}^\frown\tau$
identifies $T_n(C)$ with the subtree rooted at
$\langle n\rangle$.
Consequently,
\[
 \rho=\rk(\TT_\Gamma(C))
 =\sup_{n\in\N}\bigl(\rk(T_n(C))+1\bigr).
\]
In particular, $\rk(T_n(C))<\rho$ for every $n\in\N$.

Finally, let $T=\TT_\Gamma(C)$.
The Kleene--Brouwer ordering of $T$, with proper extensions
preceding their prefixes, is a $C$-computable well-order.
Write
\[
 \alpha=\ot(T,<_{\mathrm{KB}})<\omega_1^C.
\]
For $\sigma\in T$, let
\[
 h(\sigma)=
 \ot\bigl(\{\tau\in T:\tau<_{\mathrm{KB}}\sigma\},
 <_{\mathrm{KB}}\bigr).
\]
Every immediate successor of $\sigma$ in the tree precedes
$\sigma$ in the Kleene--Brouwer ordering.
Well-founded induction therefore gives
$\rk_T(\sigma)\leq h(\sigma)$ for every $\sigma\in T$.
In particular,
\[
 \rho=\rk_T(\varnothing)
 \leq h(\varnothing)
 <\alpha
 <\omega_1^C.
\]
\end{proof}

\begin{remark}
We do not infer \eqref{eq:rank-bound} merely from the individual
inequalities $\rk(T_n(C))<\omega_1^C$: an uncontrolled supremum
could reach $\omega_1^C$. We use a single $C$-decidable tree that is
well-founded in the ambient universe.
\end{remark}

\section{A front whose rank survives every thinning}\label{sec:fronts}

The later Ramsey arguments will repeatedly replace an infinite set
by a smaller one. We therefore need a family of finite initial
segments whose tree has high rank on every infinite restriction,
not merely on its original domain. The following construction
obtains this persistence from a single presentation of a well-order.

\begin{definition}
A front on an infinite set $V$ is a family
$\FF\subseteq[V]^{<\omega}\setminus\{\varnothing\}$ with no two
distinct members comparable under the initial-segment relation,
such that every $D\in[V]^\omega$ has an initial segment in $\FF$.
For $H\in[V]^\omega$, let
\[
\FF|H=\{F\in\FF:F\subseteq H\},
\qquad
T(\FF|H)=\{E:\exists F\in\FF|H\ (E\preceq F)\}.
\]
\end{definition}

Let $\ell:\FF\to q$ be a labeling, where $1\leq q<\omega$, and let
$c(D)=\ell(F_D)$, where $F_D$ is the unique initial segment of $D$
in $\FF$.
For every $H\in[V]^\omega$ and every $i<q$,
\[
 (\forall D\in[H]^\omega)\ c(D)=i
 \quad\Longleftrightarrow\quad
 (\forall F\in\FF|H)\ \ell(F)=i.
\]
Indeed, every $F\in\FF|H$ is the initial segment in the front of
the infinite set $F\cup H_{>\max F}$.
In particular, each homogeneous side $\HH_i$ determined by these
equivalent conditions is closed under taking infinite subsets.

The initial segment in the front is unique. The tree $T(\FF|H)$ is
well-founded: a path through it would be an infinite subset of $H$
that continues strictly beyond its own initial segment in the front.
We equip $[V]^\omega$ with the topology inherited from
$2^\N$ by identifying sets with their characteristic functions.
Every finite labeling of a front induces a clopen coloring of
$[V]^\omega$. We use the consequence of the Nash--Williams theorem
that every clopen coloring with finitely many colors has an infinite
homogeneous set, even after restriction to an infinite set $H$
\cite[Theorem 1.2]{MV}. The finite-color version follows by iterating
the binary version.

\begin{lemma}[High-rank fronts]\label{lem:high-front}
Let $V$ be an infinite $Z$-computable set, and let $W$ be a
$Z$-computable well-order of order type $\beta$. There is a front
$\FF$ on $V$, decidable in $Z$, such that
\begin{equation}\label{eq:high-front}
\forall H\in[V]^\omega\qquad \rk(T(\FF|H))\geq\beta.
\end{equation}
\end{lemma}

\begin{proof}
Adjoin a new greatest element $\top$ to $W$, obtaining
a well-order $L$.
Using a tagged coding, we may assume that the domain of $L$
is a $Z$-decidable subset of $\N$ and that $<_L$ is
decidable in $Z$.

For $a\in L$ and $x\in\N$, let
\[
 P(a,x)=\{b\in L:b\leq x\text{ in the natural order
 and }b<_La\}.
\]
This finite set can be computed in $Z$ by examining only
the natural-number codes at most $x$.
If $P(a,x)$ is nonempty, let $d(a,x)$ be its greatest
element in $<_L$.

For each initial state $a\in L$, consider the following
procedure on increasing sequences from $V$.
Start in state $a$.
When the current state is $u$ and the next element read is $x$,
stop if $P(u,x)$ is empty; otherwise move to state $d(u,x)$
and continue.
Let $\FF_a$ consist of the nonempty finite increasing
sequences from $V$ on which the first stop occurs precisely
at the last element.

Membership in $\FF_a$ is decidable in $Z$, uniformly in $a$.
Indeed, first check that the input is a nonempty increasing
sequence from $V$, and then simulate the procedure for the
length of that sequence.
Accept exactly when the first stop occurs at its final entry.
No search beyond the given finite sequence is required.

No member of $\FF_a$ properly extends another.
Moreover, an infinite increasing sequence from $V$
on which the procedure never stops would produce an infinite
strictly descending sequence of states in $L$.
Since $L$ is well-ordered, every such infinite sequence has
an initial segment in $\FF_a$.
Thus $\FF_a$ is a front on $V$.

For $a\in L$ and $H\in[V]^\omega$, write
\[
 T_{a,H}=T(\FF_a|H).
\]
For every $x\in H$, the infinite set
$\{x\}\cup H_{>x}$ has an initial segment in $\FF_a|H$
beginning with $x$. Hence $\langle x\rangle\in T_{a,H}$.
If $P(a,x)=\varnothing$, then $\langle x\rangle$
belongs to $\FF_a|H$ and is a leaf of $T_{a,H}$.
If $x\in H$ and $P(a,x)\neq\varnothing$, then after reading
$x$ the procedure is in state $d(a,x)$.
Consequently,
\[
 \{\tau:\langle x\rangle{}^\frown\tau\in T_{a,H}\}
 =
 T_{d(a,x),H_{>x}}.
\]
Indeed, the corresponding identity holds for the front
members, and taking initial segments gives the displayed
identity.
In particular,
\[
 \rk_{T_{a,H}}(\langle x\rangle)
 =
 \rk(T_{d(a,x),H_{>x}}).
\]

Let
\[
 \eta(a)=\ot(\{b\in L:b<_La\},<_L).
\]
By well-founded induction on $a$ in $L$, simultaneously
for every $H\in[V]^\omega$, we prove
\begin{equation}\label{eq:state-rank}
 \rk(T_{a,H})\geq\eta(a).
\end{equation}
If $a$ is the least element of $L$, then
$\FF_a|H=[H]^1$, so
$\rk(T_{a,H})=1\geq0=\eta(a)$.

Now assume the claim holds for all states below $a$,
and fix an arbitrary $H\in[V]^\omega$.
For each $b<_La$, choose $x\in H$ whose natural-number
value is at least the code of $b$.
Such an $x$ exists because $H$ is infinite and therefore
unbounded in $\N$.
Then $b\in P(a,x)$, so
\[
 b\leq_L d(a,x)<_La.
\]
Since $H_{>x}$ is infinite, the induction hypothesis
and the section identity above give
\[
\begin{aligned}
 \rk(T_{a,H})
 &\geq \rk_{T_{a,H}}(\langle x\rangle)+1\\
 &= \rk(T_{d(a,x),H_{>x}})+1\\
 &\geq \eta(d(a,x))+1\\
 &\geq \eta(b)+1.
\end{aligned}
\]
Taking the supremum over all $b<_La$, we obtain
\[
 \rk(T_{a,H})
 \geq
 \sup_{b<_La}\bigl(\eta(b)+1\bigr)
 =
 \eta(a).
\]
This completes the induction.

Finally, $\eta(\top)=\beta$.
Taking $\FF=\FF_\top$, we obtain a $Z$-decidable front
on $V$ satisfying \eqref{eq:high-front} for every
$H\in[V]^\omega$.
\end{proof}

\begin{remark}
The required conclusion is a lower bound, not an exact computation
of the rank. The bound holds on every infinite $H$, including sets
not computable in $Z$. We assume neither an effective system of
fundamental sequences nor a list of all well-orders.
\end{remark}

\section{From the omission rank to a clopen nucleus with an ordinary decoder}\label{sec:nucleus}

We now use the rank bound to obtain finite certificates for computing
$A$. The certificates may be recognized using $A$, but the decoder
itself must use only its given infinite subset. The proposition
separates these roles: a rank comparison proves that an appropriate
homogeneous side is nonempty, while monotonicity makes every answer
seen before the relevant certificate correct. This also permits
later constructions to leave $C$ while retaining a decoder for $A$.

\begin{proposition}[Clopen certification]\label{prop:nucleus}
Let $A$ be infinite. Suppose there exist $C\in[A]^\omega$ and an
ordinary positive operator $\Gamma$ such that
\[
\Gamma(D)=\Graph(A)\quad(D\in[C]^\omega),
\qquad \rk(\TT_\Gamma(C))<\omega_1^A.
\]
There exist a front $\GG$ on $A$, decidable in $A$, and an
$A$-computable binary labeling $\ell:\GG\to\{0,1\}$ such that
the homogeneous side
\begin{equation}\label{eq:good-side}
\HH_1=\{R\in[A]^\omega:\ell(F)=1\text{ for every }F\in\GG|R\}
\end{equation}
is nonempty. Moreover, there is an ordinary functional $\Psi$,
which does not consult $A$, $C$, or the front, such that
\begin{equation}\label{eq:ordinary-side}
\Psi^D=A\qquad(D\in\HH_1).
\end{equation}
\end{proposition}

\begin{proof}
Let $\rho=\rk(\TT_\Gamma(C))$.
Since $\rho<\omega_1^A$, fix an $A$-computable well-order
$W$ of order type $\beta$ such that
\[
 \rho<\beta<\omega_1^A.
\]
This is an existence choice of a fixed index for $W$.
We do not claim that such an index can be obtained uniformly
from $A$ and an index for $\Gamma$.

Let $\FF$ be the $A$-decidable front on $A$ supplied by
Lemma~\ref{lem:high-front}. Thus
\[
 \rk(T(\FF|K))\geq\beta
 \qquad(K\in[A]^\omega).
\]
Define
\begin{equation}\label{eq:successor-front}
 \GG=\{\{x\}\cup E:x\in A,\ E\in\FF,\ x<\min E\}.
\end{equation}
Membership in $\GG$ is decidable in $A$: a finite set $F$
belongs to $\GG$ exactly when $F\subseteq A$, $|F|\geq2$,
and $F\setminus\{\min F\}\in\FF$.
If two members of $\GG$ are comparable under the
initial-segment relation, they have the same first element;
removing that element gives two comparable members of $\FF$,
which must be equal.
Finally, every $D\in[A]^\omega$ has an initial segment in $\GG$:
take $x=\min D$ and the initial segment in $\FF$ of $D_{>x}$.
Hence $\GG$ is a front on $A$.

For $F\in\GG$, define
\begin{equation}\label{eq:leaf-color}
 \ell(F)=1\iff
 \left\{
 \begin{array}{l}
 \Gamma_{|F|}(F)\subseteq\Graph(A),\\[1mm]
 \forall n<\min F\ \exists b<2\
 ((n,b)\in\Gamma_{|F|}(F)).
 \end{array}\right.
\end{equation}
Both conditions concern a finite computable approximation,
so $\ell$ is $A$-computable.

\medskip\noindent
\textbf{Nonemptiness of side 1.}
Apply the clopen Ramsey theorem to the labeling of $\GG|C$.
There exist an infinite $H\subseteq C$ and a color $j<2$
such that $\ell(F)=j$ for every $F\in\GG|H$.
If $j=1$, then $H\in\HH_1$.

Suppose instead that $j=0$.
Fix $x\in H$ with $x>0$.
For every $E\in\FF|H_{>x}$, the set $\{x\}\cup E$
belongs to $\GG|H$ and has color $0$.
Since $\{x\}\cup E\subseteq C$, positivity gives
\[
 \Gamma_{|E|+1}(\{x\}\cup E)
 \subseteq\Gamma(C)=\Graph(A).
\]
Thus the safety condition in \eqref{eq:leaf-color} holds,
and some input below $x$ must be omitted.
Define the finite coloring
\[
 \chi_x(E)=
 \min\{n<x:(n,0),(n,1)\notin
 \Gamma_{|E|+1}(\{x\}\cup E)\}
\]
on $\FF|H_{>x}$.
This is well-defined and takes values in
$\{0,\ldots,x-1\}$.

Apply the clopen Ramsey theorem again.
We obtain $K\in[H_{>x}]^\omega$ and $n<x$ such that
$\chi_x(E)=n$ for every $E\in\FF|K$.
Consequently,
\[
 (n,0),(n,1)\notin
 \Gamma_{|E|+1}(\{x\}\cup E)
 \qquad(E\in\FF|K).
\]
By \eqref{eq:monotonicity},
\[
 \Gamma_{|E|}(E)
 \subseteq
 \Gamma_{|E|+1}(\{x\}\cup E).
\]
Hence the increasing enumeration of every $E\in\FF|K$
belongs to $T_n(C)$.
Since $T_n(C)$ is closed under initial segments,
\[
 T(\FF|K)\subseteq T_n(C).
\]
Using Lemmas~\ref{lem:high-front} and~\ref{lem:rank},
we obtain the contradiction
\begin{equation}\label{eq:contradiction}
 \beta
 \leq\rk(T(\FF|K))
 \leq\rk(T_n(C))
 <\rho<\beta.
\end{equation}
Therefore $j=1$, and $\HH_1$ is nonempty.

\medskip\noindent
\textbf{Decoding without recognizing the front.}
Define $\Psi^D(n)$ using only $D$ and the ordinary index
of $\Gamma$.
Read the elements of $D_{>n}$ in increasing order.
At iteration $k\geq1$, let $F_k$ consist of the first $k$
elements read and compute the finite set $\Gamma_k(F_k)$.
If there is a $b<2$ such that $(n,b)\in\Gamma_k(F_k)$,
return the least such $b$.
Otherwise continue to the next iteration.

This procedure uses neither $A$ nor $C$, and does not
consult $\GG$ or $\ell$.
In particular, it does not test whether $F_k$ is a leaf.

To verify correctness, fix $D\in\HH_1$ and $n\in\N$.
Since
\[
  \GG|D_{>n}\subseteq\GG|D,
\]
the infinite set $D_{>n}$ still belongs to $\HH_1$.
Let $F$ be its unique initial segment in $\GG$, and put
$m=|F|$, so that $F=F_m$ in the notation of the algorithm.

Since $\ell(F_m)=1$ and $n<\min F_m$, we have
\[
 \Gamma_m(F_m)\subseteq\Graph(A),
 \qquad
 (n,A(n))\in\Gamma_m(F_m).
\]
For every $k\leq m$, monotonicity gives
\begin{equation}\label{eq:before-certificate}
 \Gamma_k(F_k)
 \subseteq\Gamma_m(F_m)
 \subseteq\Graph(A).
\end{equation}
Thus the procedure halts by the $m$-th iteration,
and every answer available up to and including that
iteration is correct.
It follows that $\Psi^D(n)=A(n)$.
Since $D$ and $n$ were arbitrary,
$\Psi^D=A$ for every $D\in\HH_1$.
\end{proof}

\begin{remark}\label{rem:safety}
We have not proved that $\Gamma(R)$ is safe for every $R\in\HH_1$:
incorrect outputs could appear after the first leaf.
The proof uses precisely \eqref{eq:before-certificate}; the new
program halts before proceeding beyond the relevant certificate.
Simply applying $\Gamma$ to all of $R$, without this stopping rule,
would not be justified.
\end{remark}

\section{Reconstructing a solution from its color profile}\label{sec:profiles}

The remaining task is to reconstruct a selected homogeneous set,
not just to compute the original target from it. We encode the
choices needed for this reconstruction in a color profile on finite
sets. The proof has two complementary parts: the profile determines
a canonical solution, and every infinite subset of that solution,
together with the base oracle, recovers the profile. Recovery must
work on all finite-set inputs, since repeating a least-candidate
search also tests candidates outside the final solution.
The class of canonical outputs will be analytic relative to the
base oracle. The only external result used in this section is the
clopen Ramsey theorem.

\begin{theorem}[A relative UI basis for a clopen side]\label{thm:profile-basis}
Let $V$ be an infinite $Z$-computable set, let $\FF$ be a front on
$V$ decidable in $Z$, and let $\ell:\FF\to q$ be $Z$-computable,
where $1\leq q<\omega$.
Fix $i<q$ and suppose that the homogeneous side
\[
\HH_i=\{H\in[V]^\omega:\ell(F)=i\text{ for every }F\in\FF|H\}
\]
is nonempty. There exist a nonempty $\Sigma^1_1(Z)$ class
$\KK\subseteq\HH_i$ and a single ordinary functional $\Xi$ such that
\begin{equation}\label{eq:relative-basis}
\forall R\in\KK\ \forall D\in[R]^\omega\qquad
\Xi^{Z\oplus D}=R.
\end{equation}
\end{theorem}

\begin{proof}
Fix indices for the procedures computing $V$, $\FF$, and
$\ell$ relative to $Z$.
For $D\in[V]^\omega$, let $c(D)=\ell(F_D)$, where $F_D$
is the unique initial segment of $D$ in $\FF$.
For each $E\in[V]^{<\omega}$, define
\[
 c_E(S)=c(E\cup S)
 \qquad(S\in[V_{>\max E}]^\omega).
\]
The map $S\mapsto E\cup S$ is continuous, so $c_E$ is
clopen.

Call $E\in[V]^{<\omega}$ terminal if some $F\in\FF$
is an initial segment of $E$.
Terminality is decidable in $Z$ by checking the finitely
many initial segments of $E$.
The witnessing $F$ is unique, and in this case $c_E$
is constantly equal to $\ell(F)$.

\medskip\noindent
\textbf{Existence of a coherent profile.}
Fix $H\in\HH_i$.
Choose a $Z$-computable enumeration without repetition
of all finite subsets $E_s$ of $V$, with
$E_0=\varnothing$.
Put $H_{-1}=H$.
At stage $s$, apply the clopen Ramsey theorem to $c_{E_s}$
on the infinite set $(H_{s-1})_{>\max E_s}$.
Choose an infinite
\[
 H_s\subseteq(H_{s-1})_{>\max E_s}
\]
such that $c_{E_s}$ is constant on $[H_s]^\omega$,
and denote its constant value by $\tau(E_s)$.

Choose $u_s\in H_s$ so that the sequence $(u_s)$ is
strictly increasing, and let $G=\{u_s:s\in\N\}$.
For every fixed $s$, all but finitely many elements of
$G$ belong to $H_s$.
We have $\tau(\varnothing)=i$, and whenever $E$ is
terminal with initial leaf $F$, we have
$\tau(E)=\ell(F)$.

Moreover,
\begin{equation}\label{eq:coherence}
 \forall E\in[V]^{<\omega}\ \exists N\geq\max E\
 \forall x\in G\quad
 x>N\Longrightarrow\tau(E\cup\{x\})=\tau(E).
\end{equation}
To prove this, fix $E=E_s$ and choose $N\geq\max E$
such that $G_{>N}\subseteq H_s$.
For any $x\in G_{>N}$, let $t$ satisfy
$E_t=E\cup\{x\}$.
Choose an infinite tail $K$ of $G$ contained in
$H_s\cap H_t$ and lying above $x$.
Then
\[
 \tau(E)
 =c_E(\{x\}\cup K)
 =c_{E\cup\{x\}}(K)
 =\tau(E\cup\{x\}).
\]

Using a fixed coding of finite sets, extend $\tau$
by zero on finite sets not contained in $V$.
Call a pair $(\tau,G)$ admissible if $G\in[V]^\omega$,
the values of $\tau$ lie in $q$, $\tau$ is zero outside
$[V]^{<\omega}$, $\tau(\varnothing)=i$, $\tau$ has the
prescribed values on terminal sets, and
\eqref{eq:coherence} holds.
The preceding argument proves that an admissible pair
exists.
No computability in $Z$ is asserted for $\tau$, $G$,
or the sets $H_s$.
Only the stated admissibility conditions will be used below.

\medskip\noindent
\textbf{Canonical reconstruction from $Z\oplus\tau$.}
Fix any admissible pair $(\tau,G)$.
Define $(r_s)$ recursively.
At stage $s$, let
\[
 b_s=\max\bigl(\{s\}\cup\{r_j:j<s\}\bigr),
\]
and choose the least $x\in V$ with $x>b_s$ such that
\begin{equation}\label{eq:regeneration}
 \forall E\subseteq V\cap[0,b_s]\qquad
 \tau(E\cup\{x\})=\tau(E).
\end{equation}
Put $r_s=x$.
This is a finite test decidable in $Z\oplus\tau$.
For each of the finitely many sets $E$ involved,
\eqref{eq:coherence} supplies a threshold.
Every sufficiently large element of $G$ exceeds all
these thresholds and $b_s$, and hence is a candidate.
Thus the search for the least candidate terminates.

Let $R=\{r_s:s\in\N\}$.
The sequence is strictly increasing and $r_s>s$.
The reconstruction procedure uses $Z\oplus\tau$,
but does not query $G$.

For every finite $E\subseteq R$, we have $\tau(E)=i$.
Indeed, insert the elements of $E$ in the order of
their construction.
When an element $r_s$ is inserted, the previously
inserted elements lie in $V\cap[0,b_s]$, so
\eqref{eq:regeneration} preserves the value of $\tau$.
The initial value is $\tau(\varnothing)=i$.
In particular, if $E\in\FF|R$, then $E$ is terminal
and $\ell(E)=\tau(E)=i$.
Therefore $R\in\HH_i$.

\medskip\noindent
\textbf{Recovering the profile from $Z\oplus D$.}
We describe a single procedure $\Delta^{Z\oplus D}$
on finite-set inputs.
Given $E$, return zero if $E\not\subseteq V$.
If $E$ is terminal, return the color of its initial
leaf.
Otherwise put
\[
 m=\max E+1
\]
with $m=0$ when $E=\varnothing$, and read
\[
 d_j=p_D(m+j)\qquad(j\in\N).
\]
Here $m+j$ is a position in the increasing enumeration of $D$,
not a lower bound on the numerical value of the element read.
Starting with $Q_0=E$, let
\[
 Q_t=E\cup\{d_0,\ldots,d_{t-1}\}.
\]
Continue until $Q_t$ becomes terminal, and then
return the color of its initial leaf.

To verify this procedure, fix $D\in[R]^\omega$.
Only in the verification, write
\[
 d_j=r_{s_j}.
\]
Because $D$ is a subset of the range of the increasing
sequence $(r_s)$, the indices $s_j$ are strictly
increasing and satisfy
\[
 s_j\geq m+j.
\]
Hence $\max E<s_j<r_{s_j}=d_j$.
Every earlier $d_h$, for $h<j$, was chosen before
stage $s_j$.
It follows that
\[
 Q_j\subseteq V\cap[0,b_{s_j}].
\]
Applying \eqref{eq:regeneration} at stage $s_j$ gives
$\tau(Q_{j+1})=\tau(Q_j)$.
By induction,
\begin{equation}\label{eq:profile-recovery}
 \tau(E\cup\{d_0,\ldots,d_{t-1}\})=\tau(E)
 \qquad(t\in\N).
\end{equation}

The sets $Q_t$ are initial segments of the same
infinite set
\[
 E\cup\{d_j:j\in\N\}\subseteq V.
\]
Since $\FF$ is a front on all of $[V]^\omega$,
one of the $Q_t$ is terminal.
If $F$ is its initial leaf, admissibility and
\eqref{eq:profile-recovery} give
\[
 \ell(F)=\tau(Q_t)=\tau(E).
\]
Thus the procedure terminates and returns $\tau(E)$.
The same conclusion holds in the terminal and
out-of-domain cases by definition.
Therefore
\[
 \Delta^{Z\oplus D}=\tau
 \qquad(D\in[R]^\omega).
\]
This equality holds on all finite-set inputs, not merely
on finite subsets of $R$. This will allow us to repeat
the least-candidate search, including its tests on
candidates outside $R$.

The procedure $\Delta$ uses only $Z\oplus D$ and
the fixed indices for $V$, $\FF$, and $\ell$.
It does not query $R$, $G$, or $\tau$, and it does
not compute the indices $s_j$.

Now define $\Xi^{Z\oplus D}$ by repeating the
canonical reconstruction, answering each query to
$\tau$ by the procedure $\Delta^{Z\oplus D}$.
Induction on $s$ shows that this reproduces the same
$r_s$: once the earlier elements have been reproduced,
the bound $b_s$, every candidate test, and the least
successful candidate are unchanged.
To decide whether $n\in R$, generate the sequence
until the first element at least $n$ is reached.
This search terminates since $r_s>s$.
Thus
\[
 \Xi^{Z\oplus D}=R
 \qquad(D\in[R]^\omega).
\]
The index of $\Xi$ is independent of the admissible
pair and of the corresponding output $R$.

\medskip\noindent
\textbf{The analytic class.}
Let $\KK$ consist of all sets $R$ for which there
exist a profile $\tau$, an infinite set $G$, and
a sequence $r\in\N^\N$ such that $(\tau,G)$ is
admissible, $r$ satisfies the least-candidate recursion
above, and $R=\rng(r)$.

Admissibility is arithmetical in $Z,\tau,G$.
In particular, the quantifier over finite sets in
\eqref{eq:coherence} is a number quantifier.
The least-candidate recursion and the condition
$R=\rng(r)$ are also arithmetical in the displayed
parameters.
Consequently, $\KK$ is $\Sigma^1_1(Z)$.

The existence of an admissible pair proves that
$\KK$ is nonempty.
For every admissible pair, the preceding verification
shows that its output belongs to $\HH_i$ and that
the same functional $\Xi$ recovers that output from
$Z\oplus D$ for every infinite subset $D$ of the output.
Hence $\KK\subseteq\HH_i$ and
\eqref{eq:relative-basis} holds.

The front and its labeling are fixed $Z$-computable
objects.
Their validity as a front and labeling is not imposed
again in the formula defining $\KK$.
Likewise, universal correctness of $\Xi$ on all
infinite subsets is a consequence of the verification,
not an additional clause in that formula.
\end{proof}

\begin{remark}\label{rem:escape}
The new set $R$ is contained in $V$ and preserves the homogeneous
color, but we do not assert that $R\subseteq H$, where $H$ was the
initial solution. That solution is used to prove the existence of
the profile. The subsequent canonical choices may lie elsewhere
in $V$. The class $\KK$ itself need not be closed under taking
infinite subsets.
\end{remark}

\section{Composition, both parts of Q1.7, and degree control}
\label{sec:conclusion}

The two preceding constructions provide complementary computations.
On the homogeneous side obtained by certification, each infinite
subset computes $A$ without an additional oracle. The profile
construction recovers the selected set using that subset together
with $A$. Composing these procedures removes the base oracle.
We first state the resulting local criterion, whose hypotheses do
not require the ambient set to be introreducible or introenumerable;
Lemma~\ref{lem:ghpt} supplies them for every infinite introenumerable
set. The ordinal and hyperjump corollaries then select members of
the same analytic class using standard basis theorems.

\begin{theorem}[A local criterion with ordinal preservation]
\label{thm:local}
Let $A$ be infinite. Suppose that there is
$C\in[A]^\omega$ such that
\[
 A\ui C,
 \qquad
 \omega_1^C\leq\omega_1^A.
\]
Then there exist a nonempty $\Sigma^1_1(A)$ class
$\KK\subseteq[A]^\omega$ and two ordinary functionals
$\Psi,\Theta$ such that
\[
 \forall R\in\KK\ \forall D\in[R]^\omega\qquad
 \Psi^D=A
 \quad\text{and}\quad
 \Theta^D=R.
\]
In particular, every member of $\KK$ is uniformly
introreducible.
The set $A$ need not be introreducible.
\end{theorem}

\begin{proof}
The uniform reduction $A\ui C$ uniformly enumerates
$\Graph(A)$ from every infinite subset of $C$.
By the positive compilation of
\cite[Proposition 2.1]{GHPT}, fix an ordinary positive
enumeration operator $\Gamma$ such that
\[
 \Gamma(D)=\Graph(A)
 \qquad(D\in[C]^\omega).
\]
Lemma~\ref{lem:rank} gives
\[
 \rk(\TT_\Gamma(C))
 <\omega_1^C
 \leq\omega_1^A.
\]

By Proposition~\ref{prop:nucleus}, there exist an
$A$-decidable front $\GG$ on $A$, an $A$-computable
binary labeling of $\GG$, and a nonempty homogeneous
side $\HH_1$ such that a fixed ordinary functional
$\Psi$ satisfies
\[
 \Psi^D=A
 \qquad(D\in\HH_1).
\]

Apply Theorem~\ref{thm:profile-basis} with $Z=A$,
$V=A$, front $\GG$, and color $1$.
We obtain a nonempty $\Sigma^1_1(A)$ class
\[
 \KK\subseteq\HH_1\subseteq[A]^\omega
\]
and a single ordinary functional $\Xi$ such that
\[
 \Xi^{A\oplus D}=R
 \qquad(R\in\KK,\ D\in[R]^\omega).
\]

If $R\in\KK$ and $D\in[R]^\omega$, then
$D\in\HH_1$, since $\HH_1$ is closed under taking
infinite subsets.
Consequently, $\Psi^D=A$ for every such $R$ and $D$.

Define an ordinary functional $\Theta$ by simulating
$\Xi$ with a virtual joined oracle.
Queries to the second component are answered directly
using $D$.
A query at $m$ to the first component is answered by
running $\Psi^D(m)$.
If a required subcomputation diverges, the simulation
does not proceed past that query.
If it returns a value outside $\{0,1\}$, the simulation
diverges; otherwise its value is used as the answer
to the oracle query.
Thus, whenever $\Psi^D$ is total and binary-valued,
the simulation satisfies
\begin{equation}\label{eq:composition}
 \Theta^D=\Xi^{\Psi^D\oplus D}.
\end{equation}

Fix $R\in\KK$, $D\in[R]^\omega$, and $n\in\N$.
The computation $\Xi^{A\oplus D}(n)$ halts and
therefore makes only finitely many oracle queries.
Every query to its first component is replaced by
a terminating computation $\Psi^D(m)=A(m)$.
Hence the simulation follows the same computation,
with only finite delays, and returns
\[
 \Theta^D(n)=\Xi^{A\oplus D}(n)=R(n).
\]
The computation of $\Psi^D(m)$ does not invoke $\Xi$,
so this simulation is not circular.

The indices of $\Psi$ and $\Xi$ are fixed independently
of $R$ and $D$, and the same is therefore true of the
index of $\Theta$.
We have proved both required equalities for every
$R\in\KK$ and every $D\in[R]^\omega$.
The class $\KK$ itself is unchanged, so it remains
nonempty and $\Sigma^1_1(A)$.
\end{proof}

\begin{proof}[Proof of Theorem~\ref{thm:main}]
Let $A$ be infinite and introenumerable.
By Lemma~\ref{lem:ghpt}, there is $C\in[A]^\omega$ such that
\[
  A\ui C,\qquad \omega_1^C=\omega_1^A.
\]
Apply Theorem~\ref{thm:local}.
\end{proof}

\begin{corollary}[The first part of Question 1.7]\label{cor:first}
Every infinite introreducible set has an infinite uniformly
introreducible subset.
\end{corollary}
\begin{proof}
Every introreducible set is introenumerable, so
Theorem~\ref{thm:main} applies.
\end{proof}

\begin{corollary}[The second part of Question 1.7]\label{cor:second}
Every infinite introenumerable set $A$ has an infinite uniformly
introenumerable subset. In fact, for the class $\KK$ in
Theorem~\ref{thm:main}, there is a single ordinary positive enumeration
operator $\Lambda$ such that
\[
 \Lambda(D)=R
 \qquad(R\in\KK,\ D\in[R]^\omega).
\]
\end{corollary}
\begin{proof}
Let $\Theta$ be the ordinary functional supplied by
Theorem~\ref{thm:main}, so that
\[
 \Theta^D=R
 \qquad(R\in\KK,\ D\in[R]^\omega).
\]
Uniformly from the fixed functional $\Theta$, enumerate an axiom
$F_\sigma\mapsto n$ whenever $\sigma\in 2^{<\omega}$ and
a finite computation $\Theta^\sigma(n)\downarrow=1$ is found,
where
\[
 F_\sigma=\{u<|\sigma|:\sigma(u)=1\}.
\]
Here a computation with finite oracle $\sigma$ is required
to make all its oracle queries below $|\sigma|$.
This defines one ordinary positive enumeration operator
$\Lambda$, independently of $R$.

Fix $R\in\KK$ and $D\in[R]^\omega$.
If $n\in R$, the computation $\Theta^D(n)=1$ has a finite initial-segment
witness and hence produces an axiom whose support is contained in $D$.
Thus $R\subseteq\Lambda(D)$.
Conversely, if $n\in\Lambda(D)$, take a witnessing $\sigma$ with
$F_\sigma\subseteq D$ and $\Theta^\sigma(n)=1$.
The set
\[
 E=F_\sigma\cup\bigl(D\setminus[0,|\sigma|)\bigr)
\]
is infinite, is contained in $D\subseteq R$, and satisfies
\[
 E\cap[0,|\sigma|)=F_\sigma.
\]
Thus the characteristic function of $E$ extends $\sigma$, so
$\Theta^E(n)=1$.
Since $E\in[R]^\omega$, Theorem~\ref{thm:main} gives
$\Theta^E=R$. Hence $R(n)=\Theta^E(n)=1$, so $n\in R$.
This proves $\Lambda(D)\subseteq R$.
Together with the reverse inclusion proved above, we obtain
$\Lambda(D)=R$.
Since $R\in\KK$ and $D\in[R]^\omega$ were arbitrary, the same
operator $\Lambda$ works for every such pair.
\end{proof}

\begin{remark}
The uniformly introreducible conclusion is stronger as a property of
the selected witness. As existential statements for a fixed infinite
$A$, however, having a uniformly introenumerable subset and having a
uniformly introreducible subset are already equivalent by \cite[Theorem 1.4]{GHPT}.
The new step here is obtaining such a witness from a merely
introenumerable $A$ using the local criterion.
\end{remark}

\begin{corollary}[Ordinal control]\label{cor:ordinal}
Every infinite introenumerable set $A$ has an infinite UI subset
$R\subseteq A$ such that
\[
A\leq_T R,
\qquad \omega_1^R=\omega_1^A.
\]
\end{corollary}

\begin{proof}
Let $\KK\subseteq[A]^\omega$ be the nonempty
$\Sigma^1_1(A)$ class supplied by Theorem~\ref{thm:main}.
Every member of $\KK$ is uniformly introreducible, and
$\Psi^R=A$ for every $R\in\KK$.

By the relativized Gandy basis theorem, as stated immediately
after \cite[Corollary 3.13]{GHPT}, choose $R\in\KK$ such that
\[
 \omega_1^{A\oplus R}=\omega_1^A.
\]
Since $\Psi^R=A$, we have $A\leq_T R$ and hence
$A\oplus R\equiv_T R$.
Therefore, by the invariance of $\omega_1^X$ under
Turing equivalence,
\[
 \omega_1^R=\omega_1^{A\oplus R}=\omega_1^A.
\]
As $R\in\KK$, it is an infinite uniformly introreducible
subset of $A$, as required.
\end{proof}

\Needspace{9\baselineskip}
\begin{corollary}[Hyperjump control]
\label{cor:hyperjump}
For every infinite introenumerable set $A$, there is an infinite
uniformly introreducible set $R\subseteq A$ such that
\[
 A\leq_T R\leq_T \OO^A,
 \qquad
 \OO^R\equiv_T\OO^A.
\]
\end{corollary}

\begin{proof}
Let $\KK\subseteq[A]^\omega$ be the nonempty
$\Sigma^1_1(A)$ class supplied by Theorem~\ref{thm:main}.
Every member of $\KK$ is uniformly introreducible.
Taking $D=R$ in \eqref{eq:main}, we also obtain
$\Psi^R=A$ for every $R\in\KK$.

Relativizing \cite[Corollary 2.7]{JS} to the base
oracle $A$, choose $R\in\KK$ such that
\[
 \OO^{A\oplus R}
 \equiv_T R\oplus\OO^A
 \equiv_T \OO^A.
\]
Here $\OO^{A\oplus R}$ denotes the hyperjump of the
joined oracle $A\oplus R$.
In particular, $R\leq_T \OO^A$.
Moreover, $\Psi^R=A$, so $A\leq_T R$ and hence
$A\oplus R\equiv_T R$.
Since the hyperjump preserves Turing equivalence,
\[
 \OO^R\equiv_T\OO^{A\oplus R}\equiv_T\OO^A.
\]
Finally, $R\in\KK\subseteq[A]^\omega$ is infinite and
uniformly introreducible, as required.
\end{proof}

\begin{remark}
The hyperjump corollary is a separate application of a basis theorem,
not a new hyperjump inversion theorem. Both parts of Question 1.7 follow from
Theorem~\ref{thm:local} and Lemma~\ref{lem:ghpt}; the ordinal
control in Corollary~\ref{cor:ordinal} uses only the Gandy basis
theorem. The fact that $R\leq_T \OO^A$ does not imply
$R\in\HYP(A)$.
\end{remark}

\begin{remark}
The argument does not provide a uniform procedure for selecting
the decoder indices or a $\Sigma^1_1(A)$ code for the class $\KK$
from $A$. The final subset need not be contained in either of the
intermediate sets $B,C$ from Lemma~\ref{lem:ghpt}: the local criterion
reconstructs a subset of the original set $A$.
\end{remark}

\section{Conclusions and further questions}\label{sec:conclusions}

We have shown that every infinite introenumerable set contains an
infinite uniformly introreducible subset, giving affirmative answers
to both parts of \cite[Question 1.7]{GHPT}.
The proof separates uniform computation of the original set from
reconstruction of the selected subset, and then composes the two
procedures. The definability of the resulting class of witnesses
also permits separate refinements controlling computable ordinals
and hyperjumps.

A natural direction is to sharpen the complexity of the selected
subset. Can the set $R$ in Theorem~\ref{thm:main} always be chosen
with $R\in\Delta^1_1(A)$, while retaining $A\ui R$ and $R\ui R$?
For introreducible $A$, one can further ask whether $R\equiv_T A$
can always be achieved, as in \cite[Open Questions 4.17(2)]{Em}.
The ordinal and hyperjump bounds established here do not by
themselves yield either strengthening.

The argument also leaves unresolved whether every introreducible
set is finitely decomposable in the sense discussed in the
introduction; see \cite[Open Questions 4.17(4)]{Em}.
Our construction produces a suitable infinite subset without
producing a finite cover of the original set by infinite sets
$A_i$ satisfying $A\ui A_i$.
Thus the existence of a uniformly introreducible subset is settled
here independently of the finite-decomposition problem.

\section*{Acknowledgments}

The main result was discovered by ChatGPT 6 Pro (OpenAI). The author 
takes full responsibility for the mathematical content.

\end{document}